\documentclass[11pt,reqno]{amsart}
\usepackage{amssymb,amsmath,amscd}
\usepackage{latexsym,bm,bbm,mathrsfs}
\usepackage{hyperref,graphicx, enumerate}
\usepackage{mathtools}
\usepackage{stmaryrd}
\usepackage[all,pdf]{xy}
\usepackage{graphicx}
\usepackage{tikz-cd}
\usepackage{color}
\graphicspath{ {./} }

\usepackage{stackrel}
\usepackage[left=30mm, right=30mm, top=30mm, bottom=30mm]{geometry}
\usepackage[shortlabels]{enumitem}
\usepackage{ulem}
\usepackage{ dsfont }
\usepackage{verbatim}

\allowdisplaybreaks

\newtheorem{thm}{Theorem}[section]

\newtheorem*{quest*}{Question}

\newtheorem{theorem}[thm]{Theorem}
\newtheorem{cor}[thm]{Corollary}
\newtheorem{prop}[thm]{Proposition}
\newtheorem{lemma}[thm]{Lemma}

\newtheorem{remark}[thm]{Remark}

\newtheorem*{theorem*}{Theorem}

\newcommand{\bbf}{{\mathbb{F}}}

\newcommand{\bbq}{{\mathbb{Q}}}

\newcommand{\bbz}{{\mathbb{Z}}}

\newcommand{\cB}{\mathcal{B}}
\newcommand{\cC}{\mathcal{C}}

\newcommand{\cN}{\mathcal{N}}

\newcommand{\bs}[1]{\left({#1}\right)}

\newcommand{\Gal}[1]{\operatorname{Gal}\bs{{#1}}}

\newcommand{\diag}[1]{\operatorname{diag}\left({#1}\right)}
\renewcommand{\Im}{\operatorname{Im}}

\newcommand{\scrT}{\mathscr{T}}
\newcommand{\tors}{\operatorname{tors}}

\newcommand{\Lt}{\widetilde{L}}

\newcommand{\Fl}{\bbf_{\ell}}
\newcommand{\Flx}{\Fl^{\times}}
\newcommand{\Fll}{\bbf_{\ell^{2}}}
\newcommand{\Fllx}{\Fll^{\times}}
\newcommand{\cyc}{\operatorname{cyc}_{\ell}}
\newcommand{\GLl}{\operatorname{GL}_{2}\left(\Fl\right)}
\newcommand{\SLl}{\operatorname{SL}_{2}\left(\Fl\right)}
\newcommand{\GLll}{\operatorname{GL}_{2}\left(\Fll\right)}

\newcommand{\GL}[1]{\operatorname{GL}_{2}\left(\bbz/{#1}\bbz\right)}

\newcommand{\Sym}{\mathfrak{S}}
\newcommand{\Alt}{\mathfrak{A}}

\usepackage{enumitem}

\title[Stability of torsion subgroups of elliptic curves]{Stability of torsion subgroups of elliptic curves over non-Galois extensions of odd prime degree}

\author{Bo-Hae Im}
\address{Department of Mathematical Sciences, KAIST, 291~Daehak-ro, Yuseong-gu, Daejeon, 34141, South Korea}
\email{bhim@kaist.ac.kr}
\thanks{Bo-Hae Im was supported by Basic Science Research Program through the National Research Foundation of Korea(NRF) grant funded by the Korea government(MSIT)(NRF-2023R1A2C1002385, or RS-2023-NR076333).}

\author{Hansol Kim}
\address{Institute of Mathematics, Academia Sinica, 6F, Astronomy-Mathematics Building, No.~1, Sec.~4, Roosevelt Road, Taipei 10617, Taiwan}
\email{jawlang@gate.sinica.edu.tw}
\thanks{Hansol Kim was supported by NSTC grant funded by Taiwan government (MST) (No.~114-2811-M-001-036).}

\date{\today}

\begin{document}

\maketitle

\begin{abstract}
	Let $K$ be a field of characteristic $0$ and $E/K$ an elliptic curve over $K$. For a finite extension $L/K$ and a prime~$\ell$, we provide Galois-theoretic sufficient conditions on $L/K$ under which $E\left(L\right)\left[\ell^{\infty}\right] = E\left(K\right)\left[\ell^{\infty}\right]$. For a non-Galois extension $L/K$ of prime degree, we relate the growth of the $\ell^{\infty}$-torsion subgroup of $E$ under the base change $L/K$ to the image of the mod-$\ell$ cyclotomic character. In particular, In particular, we  refine Gonz{\'a}lez-Jim{\'e}nez's result by ruling out certain torsion structures for quintic non-Galois extensions $L/\bbq$.
\end{abstract}

\section{Introduction}\label{sec:intro}
Elliptic curves are main objects in number theory and algebraic geometry. A remarkable structure of elliptic curves is being abelian varieties. More precisely, for an elliptic curve $E/F$ over a field $F$, the set $E\left(F\right)$ of all $F$-rational points form an abelian group under a certain addition law. The most fundamental theorem for the structure $E\left(F\right)$ is the Mordell-Weil theorem (\cite[VIII.Theorem~6.7]{Silverman}), which states that $E\left(F\right)$ is finitely generated if $F$ is a number field. In other words, the Mordell-Weil theorem implies that $$
	E\left(F\right)\cong \mathscr{T} \oplus \bbz^{r},
$$ for some finite abelian group $\mathscr{T}$ and a non-negative integer $r$. The finite abelian group $\mathscr{T}$ and the non-negative integer $r$ are called the torsion subgroup and the rank of $E$ over $F$, respectively. The torsion subgroup $\mathscr{T}$ of $E/F$ is usually denoted by ${E\left(F\right)}_{\tors}$. 

Even over a general field $F$, for a positive integer $N$ which is not divisible by the characteristic of $F$, it is known that $$
	{E\left(F\right)}\left[N\right] \cong \bbz/m\bbz \oplus \bbz/n\bbz
$$ for some positive divisors $m, n$ of $N$ such that $m \mid n$. Building upon this foundation, the torsion subgroups of elliptic curves have been extensively studied. 

Let $E/K$ an elliptic curve defined over $K$ of characteristic $0$ and $L/K$ be a field extension.
One direction of this research focuses on the case $L=K$. Writing $d=\left[K:\bbq\right]$, the possible isomorphic classes of ${E\left(K\right)}_{\tors}$ for  elliptic curves $E/K$ defined over $K$ with degree $d\le 6$ have been completely classified:  for $d=1$ by Mazur \cite{Mazur}; for~$d=2$ by Kenku-Momose \cite{Kenku_Momose} and Kamienny~\cite{Kamienny};  for $d=3$ by Jeon-Kim-Schweizer \cite[Theorem~3.4]{JKS04}; for $d=4$ by Jeon-Kim-Park \cite[Theorem~3.6]{JKP06}; and  for $d=5, 6$ by Derickx-Sutherland \cite[Theorem~1.1]{DS17}.

On the other hand, when $K = \bbq$, the possible isomorphism classes of ${E\left(L\right)}_{\tors}$ have been determined for specific values of $\left[L:\bbq\right]$: when $\left[L:\bbq\right]=2$ by \cite{GoTo14} and \cite{Na16}; $\left[L:\bbq\right]=3$ by \cite{Na16}; $\left[L:\bbq\right]=4$ by \cite{Ch16}, $\left[L:\bbq\right]=5$ by \cite{Go17}; and for prime degrees $\left[L:\bbq\right] \ge 7$ by \cite{GoNa20}.

If $K$ is a general number field, certain conditions under which  the minimal prime divisor of $\left[L:K\right]$ guarantees that ${E\left(L\right)}_{\tors} = {E\left(K\right)}_{\tors}$ are studied in \cite{G24} and \cite{IK24}.

In the studies mentioned above, ${E\left(L\right)}_{\tors}$ is described mainly in terms of the extension degrees $\left[K:\bbq\right]$ and $\left[L:K\right]$. There is another approach which considers the Galois theory of $L/K$: for example, \cite{ChDKNa21} considers the case when $K=\bbq$ and $L$ is the $\bbz_{p}$-extensions of $\bbq$, and \cite{Ch19} studies the case when $K=\bbq$ and $L$ is the maximal abelian extensions of $\bbq$.

In this paper, we investigate ${E\left(L\right)}_{\tors}$ under more general assumptions on  $K$ and $L/K$. Let $\Lt/K$ denote the Galois closure of $L/K$. We show that, under specific conditions on $\Gal{\Lt/K}$ and $\Gal{\Lt/L}$, one has $$
	{E\left(L\right)}_{\tors} = {E\left(K\right)}_{\tors}.
$$ 

Unlike prior research which primarily focuses on extensions over $\bbq$ or on low-degree number fields, our results in this paper are applicable to arbitrary extensions  $L/K$ of characteristic~$0$.
Moreover, rather than describing torsion growth solely in terms of extension degrees, our work identifies explicit Galois-theoretic conditions under which the torsion subgroup does not grow, thereby clarifying precisely when torsion remains unchanged and distinguishing our approach from existing degree-based classifications.

Our main result  in this paper is as follows:

\begin{theorem}\label{thm:non-Gal_p}
	Let $K$ be a field of characteristic $0$ and $E/K$ an elliptic curve over $K$. Let $L/K$ be a non-Galois extension of prime degree $p\ge 3$. Then the following holds for any prime~$\ell$ if $p=3$, and for any prime $\ell\neq p$ if $p\neq 3$: \begin{enumerate}[{\normalfont (a)}]
		\item  If $E\left(K\right)\left[\ell\right] \ne \left\{O\right\}$,  then $$
			E\left(L\right)\left[\ell^{\infty}\right] = E\left(K\right)\left[\ell^{\infty}\right].
		$$
		\item Let $\ell\ge 3$ be a prime with $ E\left(K\right)\left[\ell\right]= \left\{O\right\}$. If $\Im \cyc \neq \left\{\pm 1\right\}$, then $$
			E\left(L\right)\left[\ell\right] =\left\{O\right\}= E\left(K\right)\left[\ell\right],
		$$ where $\cyc$ is the mod-$\ell$ cyclotomic character of the absolute Galois group of $K$.
		\item  If $p \ge 5$ and $E\left(K\right)\left[2\right] = \left\{O\right\}$, then $$
			E\left(L\right)\left[2\right] =\left\{O\right\}= E\left(K\right)\left[2\right].
		$$
	\end{enumerate}
\end{theorem}

Before stating our next result, we note that our analysis, which leverages the properties of non-Galois quintic extensions, allows us to refine the classification of torsion subgroups over quintic extensions, given by Gonz{\'a}lez-Jim{\'e}nez~\cite[Theorem~1]{Go17}. In particular, certain torsion structures that can occur over general quintic extensions can be excluded in the non-Galois case by our result, Theorem~\ref{thm:non-Gal_p} when $K=\bbq$ and $p=5$, refining \cite[Theorem~1]{Go17} as follows:

\begin{cor}[{Refinement of \cite[Theorem~1]{Go17}}]\label{cor:ref_Go17}
	Let $E/\bbq$ an elliptic curve over $\bbq$  and $L/\bbq$ a non-Galois quintic extension. Then, the torsion subgroup $E\left(L\right)_{\tors}$ is isomorphic to one of the following: \begin{align*}
		\bbz/n\bbz, &\quad \text{ for } n \in \left\{1,2, \ldots, 10, 12\right\}, \text{ or }\\
		\bbz/2\bbz \oplus \bbz/2n\bbz, &\quad \text{ for } n \in \left\{1,2,3,4\right\}.
	\end{align*}
\end{cor}

Another application of Theorem~\ref{thm:non-Gal_p} is that it provides a single, unified proof of a refinement of the classification of torsion subgroups over non-Galois cubic fields originally obtained through three separate results. Najman~\cite[Theorem~1]{Na16} classified the possible torsion subgroups over cubic extensions. For the case of non-Galois cubic fields, two specific torsion subgroups $\bbz/13\bbz$ and $\bbz/2\bbz \oplus \bbz/14\bbz$  are ruled out by \cite[Lemma~2.2]{JS20} and \cite[Theorem~1.2]{BrNa17}, respectively. Theorem~\ref{thm:non-Gal_p} rules out both groups at once, providing a streamlined and conceptually unified independent proof  given in Section~\ref{sec:application} of the following corollary.

\begin{cor}[{\cite[Theorem~1]{Na16}, \cite[Theorem~1.2]{BrNa17}, \cite[Lemma~2.2]{JS20}}]\label{cor:ref_Na16}
	Let $E/\bbq$ an elliptic curve over $\bbq$  and $L/\bbq$ a non-Galois cubic extension. Then, the torsion subgroup $E\left(L\right)_{\tors}$ is isomorphic to one of the following: \begin{align*}
		\bbz/n\bbz, &\quad \text{ for } n \in \left\{1,2, \ldots, 10, 12, 14, 18\right\}, \text{ or }\\
		\bbz/2\bbz \oplus \bbz/2n\bbz, &\quad \text{ for } n \in \left\{1,2,3,4\right\}.
	\end{align*}
\end{cor}

\

To prove Theorem~\ref{thm:non-Gal_p}, let $L/K$ be a field extension of characteristic~$0$, $E/K$ an elliptic curve, and $\ell$ a prime. We establish Galois-theoretic conditions on $L/K$ and $\ell$ ensuring $$
	E\left(L\right)\left[\ell^{\infty}\right] = E\left(K\right)\left[\ell^{\infty}\right].
$$ 
The conditions differ substantially depending on whether $E\left(K\right)\left[\ell\right]$ is trivial:

\begin{itemize}
	\item If $E\left(K\right)\left[\ell\right]\neq \left\{O\right\}$, the sufficient condition is given in Corollary~\ref{cor:ell_non_zero} (Section~\ref{sec:non_zero}).
	\item If $E\left(K\right)\left[\ell\right] = \left\{O\right\}$, the situation is quite different, and the sufficient conditions are provided in Proposition~\ref{prop:non-Gal_3} and Proposition~\ref{prop:non-Gal_prime} (Section~\ref{sec:zero}).
\end{itemize}

The proof of  Theorem~\ref{thm:non-Gal_p} will be given in Section~\ref{sec:zero}, where we establish the conditions on $L/K$ that guarantee the non-growth of the rational $\ell$-primary torsion part.

\

In Section~\ref{sec:additional}, using the lemmas developed for the proof  of Theorem~\ref{thm:non-Gal_p}, we also establish two additional results, namely the following Theorem~\ref{thm:thm2} and Theorem~\ref{thm:Sn}.

\begin{theorem}\label{thm:thm2}
	Let $K$ be a field of characteristic $0$ and $E/K$ an elliptic curve defined over $K$. Let $L/K$ be an extension of $K$. For a positive integer $m$, if $\ell^{m} \ge 3$ and $E\left(K\right) \supseteq \left(\bbz/\ell^{m}\bbz\right)^{2}$, then for every $k\ge0$,  we have $$
		E\left(K\left(\scrT_{k}\right)\right)\left[\ell^{\infty}\right] = \scrT_{k},
	$$ where $\scrT_{k} = \left\{P \in E\left(K\right)\left[\ell^{\infty}\right]:\ell^{k}P \in E\left(K\right)\right\}$.
\end{theorem}

\begin{theorem}\label{thm:Sn}
	Let $K$ be a field of characteristic $0$, and $f \in K\left[x\right]$ a monic polynomial of degree $n\ge5$, with zeros $\alpha_{1},\alpha_{2},\cdots,\alpha_{n}$ in an algebraic closure of~$K$. If  the Galois group  $\Gal{f/K}$   of $f$ over $K$  is isomorphic to the symmetric group $\Sym_{n}$ of degree $n$, then for an extension $L/K$ adjoining any $n-2$ of the zeros, $\alpha_i$, say (after reordering) $L := K\left(\alpha_{3},\alpha_{4},\cdots,\alpha_{n}\right)$, then the torsion subgroup of $E$ does not grow over~$L$, i.e., $$
		E\left(L\right)_{\tors} = E\left(K\right)_{\tors}.
	$$
\end{theorem}

\begin{remark}\label{almostall}
	\cite[Theorem 5.40]{Milne} established that when $K$ is a number field, `almost all' polynomials over~$K$ satisfy the condition of Theorem~\ref{thm:Sn}. For the precise meaning of `almost~all', the reader is referred to \cite[Section~9]{Serre97}, {\cite[Section~3]{Serre08}}, and \cite[Proposition~2.2]{IK22}.
\end{remark}

Over the base fields $K=\bbq$ or $\bbq\left(\mu_{\infty}\right)$, the condition $\Im \cyc  \ne \left\{\pm1\right\}$ in Theorem~\ref{thm:non-Gal_p} is always satisfied (except only for $K=\bbq$ and $\ell=3$). Hence, we have the following Corollary~\ref{cor:overQ_and_Qmu} of Theorem~\ref{thm:non-Gal_p} which are proven in Section~\ref{sec:application}.

\begin{cor}\label{cor:overQ_and_Qmu}
	Let $K=\bbq$ (resp. $K=\bbq\left(\mu_{\infty}\right)$). Let $E/K$ be an elliptic curve  defined over $K$ and $L/K$  a non-Galois extension of prime degree $p\ge 3$. If $\ell\ge 5$ for $K=\bbq$ (resp. $\ell\geq 3$ for $K=\bbq\left(\mu_{\infty}\right)$) is a prime such that  $\ell \ne p$ or $E\left(K\right)\left[\ell\right] \ne \left\{O\right\}$, then $$
		E\left(L\right)\left[\ell^{\infty}\right] = E\left(K\right)\left[\ell^{\infty}\right].
	$$
\end{cor}

\subsection{Notations}\label{subsec:notations}
Let $E/K$ be an elliptic curve defined over a field $K$ of characteristic $0$. Throughout this paper, let $\overline{K}$ denote an algebraic closure of $K$, and let $\ell$ denote a prime. For a non-negative integer $m$, we denote a point of order $\ell^m$ in $E\left(\overline{K}\right)$ by $P_{m}, Q_m, T_m,$ etc., depending on context. Given such  a point, say $P_{m}$, we define $P_{k} := \ell^{m-k} P_{m}$ for a non-negative integer $k\leq m$, so that $P_k$ is of order~$\ell^k$. For convenience, we write $E[\ell^\infty]=E\left(\overline{K}\right)[\ell^\infty]$, and similarly for other torsion subgroups, omitting the base field when it is $\overline{K}$.

Denote the absolute Galois group  $\Gal{\overline{K}/K}$  of $K$ by $G_{K}$. Also, for an integer $m\geq 1$, denote by $\zeta_m$ a primitive $m$th root of unity.

Also, throughout this paper, the following notations will appear frequently. \begin{itemize}
	\item $U := \begin{pmatrix}1&1\\0&1\end{pmatrix} \in \GLl$.
	\item $\cyc$: the mod-$\ell$ cyclotomic character from $G_{K}$ to $\Flx$ defined by $\zeta_{\ell}^{\sigma} = \zeta_{\ell}^{\cyc\left(\sigma\right)}$, for $\sigma\in G_{K}$.
	\item $\diag{a,d} := \begin{pmatrix}a&0\\0&d\end{pmatrix}$ for $a,d \in \Fl$.
	\item $I_{2} := \diag{1,1}$.
	\item $\Sym_{n}$: the symmetric group of degree $n$
	\item $\Alt_{n}$: the alternating group of degree $n$
	\item For a basis $\cB=\left\{P_{1},Q_{1}\right\}$ of $E\left[\ell\right]$, we let $r_{\cB}: \Gal{K\left(E\left[\ell\right]\right)/K} \to \GLl$ be a group homomorphism defined by $$
		r_{\cB}\left(\sigma\right)\begin{pmatrix}P_{1}\\Q_{1}\end{pmatrix}
		= \begin{pmatrix}P_{1}^{\sigma}\\Q_{1}^{\sigma}\end{pmatrix}.
	$$
\end{itemize}

\begin{remark}\label{remark:fixedQ}
	Let $\cB=\left\{P_{1},Q_{1}\right\}$ be a basis  of $E\left[\ell\right]$.
\begin{enumerate}[\normalfont (a)]
	\item By the definition of $r_{\cB}$, we have that $$
			r_{\cB} \left( \Gal{K\left(E\left[\ell\right]\right)/K\left(Q_{1}\right)} \right) \subseteq \left\{\begin{pmatrix}a&b\\0&1\end{pmatrix}\in \GLl\right\}.
		$$
	\item {$($\cite[Ch.III.Proposition~8.1]{Silverman}$)$}\label{remark:Weil}
		Via the Weil pairing, we have $$
			\det \;\circ\; r_{\cB} = \cyc.
		$$
	%
	\item For two bases $\cB=\left\{P_{1},Q_{1}\right\}$ and $\mathcal{B}'=\left\{P'_{1},Q'_{1}\right\}$ of $E\left[\ell\right]$, there exists $T \in \GLl$ such that $\begin{pmatrix}P'_{1}\\Q'_{1}\end{pmatrix} = T \begin{pmatrix}P_{1}\\Q_{1}\end{pmatrix}$. It follows that $$
		r_{\mathcal{B}'} = T \;r_{\cB}\; T^{-1}.
	$$
\end{enumerate}
\end{remark}

\section{Auxiliary lemmas for torsion stability}\label{sec:non_zero}


In this section, we establish Galois-theoretic sufficient conditions on $L/K$ and $\ell$ eusuring $E\left(L\right)\left[\ell^{\infty}\right] = E\left(K\right)\left[\ell^{\infty}\right]$, when $E\left(K\right)\left[\ell\right] \ne \left\{O\right\}$. Lemma~\ref{lem:ell_sq} is particularly useful in deriving such conditions, when $E\left(K\right)\left[\ell\right] \cong \left(\bbz/\ell\bbz\right)^{2}$.

\begin{lemma}\label{lem:ell_sq}
	Let $K$ be a field of characteristic $0$ and $E/K$ be an elliptic curve defined over $K$. For a non-negative integer $r$, define a subgroup of $E\left[\ell^{\infty}\right]$, $$
		\scrT_{r} := \left\{T \in E\left[\ell^{\infty}\right]: \ell^{r}T \in E\left(K\right) \right\},
	$$ and let $K_{r}$ be the field of definition $K\left(\scrT_{r}\right)$. Then we have the following: \begin{enumerate}[{\normalfont (a)}]
		\item $K_{r}/K$ is a Galois extension.

		\item Let $m\ge 0$ and $0 \le k \le m$ be integers such that $E\left(K\right)\left[\ell^{m}\right] \cong \bbz/\ell^{k}\bbz \oplus \bbz/\ell^{m}\bbz$. Then, for each integer $r \ge 0$, there exists a basis $\left\{P_{m+r},Q_{m+r}\right\}$ of $E\left[\ell^{m+r}\right]$ such that $E\left(K\right)\left[\ell^{m}\right] = \left\langle P_{k}, Q_{m} \right\rangle$. If $E\left(K\right)\left[\ell^{\infty}\right] = E\left(K\right)\left[\ell^{m}\right]$, then $\scrT_{r} = \left\langle P_{k+r}, Q_{m+r} \right\rangle$.
		
		\item For an integer $r\ge 1$ and $T \in \scrT_{r}$ such that $\ell^{r-1}T \notin E\left(K\right)$, if $E\left(K\right)\left[\ell^{\infty}\right] \supseteq \left(\bbz/\ell^{r}\bbz\right)$, then the field extension $K\left(T\right)$ over $K$ is a Galois extension whose Galois group is isomorphic to a subgroup of exponent $\ell^{r}$ in $\left(\bbz/\ell^{r}\bbz\right)^{2}$.

		\item We assume that $\scrT_{1} \supsetneq E\left(K\right)\left[\ell^{\infty}\right]$. If $E\left(K\right)\left[\ell^{\infty}\right] \supseteq \left(\bbz/\ell^{r}\bbz\right)$ for an integer~$r\geq 1$, then the Galois group of  the extension $K_{r}$ over $K$ is an abelian group of exponent $\ell^{r}$.
		
		\item For an integer $r \ge 1$ such that $E\left(K\right) \supseteq \left(\bbz/\ell^{r}\bbz\right)^{2}$, if at least one of $r \ge 2$, $\ell \ge 3$, or $E\left(K\right) \supseteq \left(\bbz/\ell^{r+1}\bbz\right)^{2}$ holds,  then  $E\left(K_{r}\right)\left[\ell^{\infty}\right] = \scrT_{r}$.
	\end{enumerate}
\end{lemma}
\begin{proof}
	(a) For $\sigma\in G_{K}$ and $T \in \scrT_{r}$, we have $\ell^{r} \left(T^{\sigma}\right) = \left(\ell^{r} T\right)^{\sigma} = \ell^{r} T \in E\left(K\right)$. So $\scrT_{r}$ is $G_{K}$-invariant.
	
	(b) First,  by \cite[Lemma~3.1]{IK22} and the given assumption, there exists a basis $\left\{P_{m+r},Q_{m+r}\right\}$ of $E\left[\ell^{m+r}\right]$ such that $E\left(K\right)\left[\ell^{m}\right] = \left\langle P_{k}, Q_{m} \right\rangle$ for some integer $0 \le k \le m$. If $E\left(K\right)\left[\ell^{\infty}\right] = E\left(K\right)\left[\ell^{m}\right]$, we aim to show that $\scrT_{r} = \left\langle P_{k+r}, Q_{m+r} \right\rangle$. Since $\ell^{r} P_{k+r}=P_k$ and  $\ell^{r} Q_{m+r}=Q_{m}$, it is immediate that $P_{k+r}, Q_{m+r} \in \scrT_{r}$, so it is enough to show that $\scrT_{r}$ is generated by $P_{k+r}$ and $Q_{m+r}$. If $T \in \scrT_{r}$, then for some $a,b \in \bbz$, $\ell^{r} T = a P_{k} + bQ_{m}=a\ell^{r} P_{k+r}+b\ell^{r} Q_{m+r}$. Thus, $T - \left( a P_{k+r} + b Q_{m+r} \right) \in E\left[\ell^{r}\right] = \left\langle P_{r}, Q_{r} \right\rangle \subseteq \left\langle P_{k+r}, Q_{m+r} \right\rangle$. This implies that  $T\in \left\langle P_{k+r}, Q_{m+r} \right\rangle$.
	
	(c) To consider the extension $K\left(T\right)/K$, we define a function $\kappa_{T}:G_{K}\rightarrow E\left[\ell^{r}\right]$ by \begin{equation}\label{eqn:kappa}
		\kappa_{T}\left(\sigma\right) = T^{\sigma} - T, \text{ for } \sigma\in G_{K}.
	\end{equation}
	
	Since $\ell^{r} T\in E\left(K\right)$, it follows that for every $\sigma\in G_{K}$, $\ell^{r} T$ is fixed under $\sigma$, so $$
		\ell^{r} \kappa_{T}\left(\sigma\right) = \ell^{r} T^\sigma - \ell^{r} T=O,
	$$ which implies that $\kappa_{T}(\sigma)\in E[\ell^{r}]$. Hence, $\kappa_{T}$ is well-defined. Next, we show that $\kappa_{T}$ is a group homomorphism. Since $\Im(\kappa_{T})\subseteq E[\ell^{r}]\subseteq E\left(K\right)$ by assumption, it follows that  for any $\sigma, \tau \in G_{K}$, $$
		\kappa_{T}\left(\sigma\tau\right) = T^{\sigma\tau} - T = \left(T^{\sigma\tau} - T^{\tau}\right) + \left(T^{\tau} - T\right) = \left(\kappa_{T}\left(\sigma\right)\right)^{\tau} + \kappa\left(\tau_{T}\right) = \kappa_{T}\left(\sigma\right) + \kappa_{T}\left(\tau\right),
	$$ which proves that $\kappa_{T}$ is a group homomorphism. Clearly, the kernel of $\kappa_{T}$ is $\Gal{\overline{K}/K\left(T\right)}$, which implies that $K\left(T\right)$ is a Galois extension over $K$. Therefore, \begin{align}\label{eq:exponent}
		\Gal{K\left(T\right) / K} \cong G_{K}/\Gal{\overline{K}/K\left(T\right)}\cong \Im{\kappa_{T}} \subseteq\left(\bbz/\ell^{r}\bbz\right)^{2}.
	\end{align}
		
	To determine the exponent of $\Gal{K\left(T\right) / K}$, let $L := K\left(T\right)$ and consider the injective group homomorphism $\overline{\kappa_{T}}: \Gal{L/K} \to E\left[\ell^{r}\right]$ defined by $\overline{\kappa_{T}}\left(\sigma\right) = T^{\sigma} - T$, as a restriction of~$\kappa_{T}$. Referring to \eqref{eq:exponent}, the exponent of $\Gal{L/K}$ is not $\ell^{r}$ if and only if $\overline{\kappa_{T}}\left(\Gal{L/K}\right) \subseteq E\left[\ell^{r-1}\right]$. Moreover, since $\ell^{r-1}T^{\sigma} - \ell^{r-1}T = \ell^{r-1} \overline{\kappa_{T}}\left(\sigma\right)$, we have $$
		\overline{\kappa_{T}}\left(\Gal{L/K\left(\ell^{r-1}T\right)}\right) = \overline{\kappa_{T}}\left(\Gal{L/K}\right) \cap E\left[\ell^{r-1}\right].
	$$ Thus, the exponent of $\Gal{L/K}$ is not $\ell^{r}$ if and only if $\overline{\kappa_{T}}\left(\Gal{L/K}\right) = \overline{\kappa_{T}}\left(\Gal{L/K\left(\ell^{r-1}T\right)}\right)$. Since $\overline{\kappa_{T}}$ is injective,  this forces $\ell^{r-1} T \in E\left(K\right)$, contradicting our assumption on $T$.

	(d) Note that for every positive integer $k$, $E\left(K\right)\left[\ell^{\infty}\right]$ is a subgroup of $\scrT_{k}$, and  $\scrT_{k}\subseteq \scrT_{k+1}$. We first show that there exists $d \in \left\{1,2\right\}$ such that, for any positive integer $k$, the non-trivial quotient group $$
		\overline{\scrT}_{k} := \scrT_{k} / \left(E\left(K\right)\left[\ell^{\infty}\right]\right)
	$$ is isomorphic to $\left(\bbz/\ell^{k}\bbz\right)^{d}$. We argue by induction on $k$. For $k=1$, since $E\left[\ell^{s}\right] \cong \left(\bbz/\ell^{s}\bbz\right)^{2}$ for every integer $s \ge1$, and $\overline{\scrT}_{1}$ has exponent~$1$, it follows that there exists $d \in \left\{1,2\right\}$ such that $$
		\overline{\scrT}_{1} \cong \left(\bbz/\ell\bbz\right)^{d}.
	$$

	Now consider  two subgroups $A_{1} \subseteq A_{2} \subseteq \left(\bbz/\ell^{k+1}\bbz\right)^{2}$. Then there exist non-negative integers $e_{ij}$ for $i,j \in \left\{1,2\right\}$ such that $$
		e_{i1}\le e_{i2}, \quad e_{1j} \le e_{2j},
		\quad \text{ and } \quad A_{i} \cong \bbz/\ell^{e_{i1}}\bbz \oplus \bbz/\ell^{e_{i2}}\bbz
	$$ for $i,j \in \left\{1,2\right\}$. If $A_1$ and $A_2$ satisfy that\begin{equation}\label{eqn:c}
	A_{1} = A_{2}[\ell^{k}], \text{ and } A_{1} = \left\{\ell a: a \in A_{2}\right\},
	\end{equation} then for $j=1,2$, we obtain $e_{1j} = \min \left\{k,e_{2j}\right\} = \max\left\{e_{2j}-1,0\right\}$. 
	In particular, if $e_{1j}=0$, then $e_{2j} =0$; and if $e_{1j}=k$, then $e_{2j}=k+1$. Equivalently, for $c\in \left\{0,1,2\right\}$, if $A_{1} \cong \left(\bbz/\ell^{k}\bbz\right)^{c}$, then $A_{2} \cong \left(\bbz/\ell^{k+1}\bbz\right)^{c}$. It is easy to see that $A_{1} = \overline{\scrT}_{k}$ and $A_{2} = \overline{\scrT}_{k+1}$ satisfy \eqref{eqn:c}. Therefore, since  $\overline{\scrT}_{1} \cong \left(\bbz/\ell\bbz\right)^{d}$, the induction shows that $$
		\overline{\scrT}_{k} \cong \left(\bbz/\ell^{k}\bbz\right)^{d}, \text{ for every integer } k\geq 1.
	$$ 

	Now, suppose $E\left(K\right)\left[\ell^{\infty}\right] \supseteq \left(\bbz/\ell^{r}\bbz\right)$ for an integer~$r\geq 1$. Then  the above observation implies that $\overline{\scrT}_{r}$ is generated by one or two elements of order $\ell^{r}$. Equivalently,  $\scrT_{r}$ is generated by $E\left(K\right)\left[\ell^{\infty}\right]$ together with  at most two points $V$ and $W$ satisfying $$
		\langle V\rangle \cap \langle W\rangle \subseteq E\left(K\right), \quad  \ell^{r} V, \ell^{r}W\in E\left(K\right), \quad   \ell^{r-1} V, \ell^{r-1}W \notin E\left(K\right).
	$$ If $\scrT_{r}$ is generated by $E\left(K\right)\left[\ell^{\infty}\right]$ and~$V$, then $\Gal{K_{r}/K} = \Gal{K\left(V\right)/K}$ has exponent $\ell^{r}$ by~(c). If instead $\scrT_{r}$ is generated by $E\left(K\right)\left[\ell^{\infty}\right]$, $V$, and~$W$, then $\Gal{K_{r}/K} = \Gal{K\left(V,W\right)/K}$ embeds into $\left(E\left[\ell^{r}\right]\right)^{2}$ via the group homomorphism of $\left(\kappa_{V},\kappa_{W}\right): G_{K} \to \left(E\left[\ell^{r}\right]\right)^{2}$,  where the group homomorphisms $\kappa_{V}$ and $\kappa_{W}$ are defined by~\eqref{eqn:kappa}. By~(c), the images of $\kappa_{V}$ and $\kappa_{W}$ have exponent $\ell^{r}$. Hence, $\Gal{K_{r}/K}$ has exponent $\ell^{r}$.
	
	(e) The inclusion  $E\left(K_{r}\right)\left[\ell^{\infty}\right] \supseteq \scrT_{r}$  is immediate by definition of $K_{r} = K\left(\scrT_{r}\right)$. Conversely, suppose there exists a point $V \in E\left(K_{r}\right)\left[\ell^{\infty}\right]$ such that $\ell^{r}V \notin E\left(K\right)$. Since $V$ is a torsion point, there exists  the smallest positive integer $s$ such that $\ell^s V\in E\left(K\right)$. Necessarily, $s\geq r+1$. So, replacing $V$ with $\ell^{s-r-1}V$, we may assume that $V \in E\left(K_{r}\right)\left[\ell^{\infty}\right]$ such that $\ell^{r+1}V \in E\left(K\right)$, but $\ell^{r}V \notin E\left(K\right)$.
	
	Then, since $\ell^{r+1}V \in E\left(K\right)$, for any $\sigma \in \Gal{K_{r}/K}$, $R:=V^{\sigma}- V \in E\left[\ell^{r+1}\right]$. Moreover, since $\ell R \in E\left[\ell^{r}\right]\subseteq E\left(K\right)$ by assumption, we have $S:= R^{\sigma} -R \in E\left[\ell\right]$.
	
	If $E\left(K\right) \supseteq \left(\bbz/\ell^{r+1}\bbz\right)^{2}$, then $R\in E\left(K\right)$, so $S=O$, and it then follows by induction that for each positive integer $i$, $$
		V^{\sigma^{i}} = V + iR.
	$$
	
	If  $r \ge 2$ or $\ell \ge 3$, then since $S \in E\left[\ell\right] \subseteq E\left(K\right)$, an induction shows that for each positive integer~$i$, $$
		V^{\sigma^{i}} = V + iR + \dfrac{i(i-1)}{2} S.
	$$

	Now, note that $\ell \mid \frac{\ell^{r}\left(\ell^{r}-1\right)}{2}$ if $r \ge 2$ or $\ell \ge 3$. Since  the exponent of $\Gal{K_{r}/K}$ is $\ell^{r}$ by~(d),  we conclude that for each $\sigma \in \Gal{K_{r}/K}$, $$
		V = V^{\sigma^{\ell^{r}}}
		= \begin{cases}
			V + \ell^{r}R, &\text{ if } E\left(K\right) \supseteq \left(\bbz/\ell^{r+1}\bbz\right)^{2},\\
			V + \ell^{r}R + \frac{\ell^{r}\left(\ell^{r}-1\right)}{2} S = V + \ell^{r}R, &\text{ if } r \ge 2 \text{ or }\ell \ge 3.\end{cases}
	$$ Thus, in either case, we conclude that $\ell^{r} R = O$, and hence we have $\ell^{r} V^{\sigma} = \ell^{r} V$ for all $\sigma \in \Gal{K_{r}/K}$, implying that $\ell^{r} V\in E\left(K\right)$, which contradicts our assumption.
\end{proof}

From now on, we establish Galois-theoretic sufficient conditions on $L/K$ and $\ell$ ensuring  $E\left(L\right)\left[\ell^{\infty}\right] = E\left(K\right)\left[\ell^{\infty}\right]$, when $E\left(K\right)\left[\ell\right] \cong \bbz/\ell\bbz$. As noted in Remark~\ref{remark:fixedQ}, the following Lemma~\ref{lem:Borel}, which concerns subgroups of Borel subgroups in $\GLl$, plays a necessary role.

\begin{lemma}\label{lem:Borel}
	Let $G \subseteq \left\{\begin{pmatrix}a&b\\0&d\end{pmatrix} \in \GLl\right\}$ be a subgroup such that  $\ell \mid |G|$. Then for a subgroup~$H$ of $G$, $H$ is normal in $G$ if and only if either $U= \begin{pmatrix}1&1\\0&1\end{pmatrix}\in H$, or else, for every $\begin{pmatrix}a&b\\0&d\end{pmatrix}\in H$, we have $a=d \in \Flx$.
\end{lemma}

\begin{proof}
	Suppose $H$ is a normal subgroup of $G$. If there exists  $A = \begin{pmatrix}a&b\\0&d\end{pmatrix}\in H$ such that $a\neq d$ in $\Flx$, then a direct computation shows that,
	letting $1\leq \left(d-a\right)^{-1}\leq \ell-1$ as an element in~$\Flx$, we have that  $U = U^{d\left(d-a\right)^{-1}} \begin{pmatrix}a&b\\0&d\end{pmatrix} U^{-d\left(d-a\right)^{-1}} \begin{pmatrix}a&b\\0&d\end{pmatrix}^{-1}\in H$.
	
	Conversely, suppose $U\in H$. Then for any $\begin{pmatrix}a&b\\0&d\end{pmatrix} \in H$, letting $1\leq a^{-1}\leq \ell-1$ as an element in $\Flx$,  we have that
	\begin{align}\label{eqn:gen}
	\begin{pmatrix}a&b\\0&d\end{pmatrix} = \begin{pmatrix}a&0\\0&d\end{pmatrix} U^{ba^{-1}},
	\end{align}
	which implies that $H=\left\langle U, \Delta \right\rangle$, where $\Delta := \left\{\diag{a,d} \in H\right\}$.
	If $B\in G$ and $\begin{pmatrix}a&b\\0&d\end{pmatrix} \in H$, a direct computation implies that $$
		B \begin{pmatrix}a&b\\0&d\end{pmatrix} B^{-1} = \begin{pmatrix}a&b'\\0&d\end{pmatrix} = \begin{pmatrix}a&b\\0&d\end{pmatrix}U^{\left(b'-b\right)a^{-1}}\in H.
	$$ Therefore, $H$ is normal in $G$.
	
	If every element of $H$ has identical diagonal entries, then $H$ is a subgroup of  the kernel of~$f$, where  $f : G \to \Flx$ is the group homomorphism defined by $\begin{pmatrix}a&b\\0&d\end{pmatrix} \mapsto ad^{-1}$. Thus, it suffices to show that  $H$ is a characteristic subgroup of $\ker\left(f\right)$.  By \eqref{eqn:gen}, the group $\left\langle U, \diag{a,a}:a\in\Flx\right\rangle$ contains $\ker\left(f\right)$. This group is cyclic, since $U$ and $\diag{a,a}$ commute, and it is generated by two cyclic groups $\left\langle U\right\rangle$ and $\left\{\diag{a,a}:a\in\Flx\right\}$, whose orders are relatively prime. Therefore, its subgroup $\ker\left(f\right)$ is  also cyclic, and hence, $H$ is a characteristic subgroup of $\ker\left(f\right)$.
\end{proof}

\begin{prop}\label{prop:ell_linear}
	Let $K$ be a field of characteristic $0$ and $E/K$ an elliptic curve defined over~$K$ with a point $Q_{1}$ of order $\ell$ such that $E\left(K\right)\left[\ell\right] = \left\langle Q_{1} \right\rangle$. Let $P_{1}\in E\left[\ell\right]$ be a point of order $\ell$ which is linearly independent from $Q_{1}$. \begin{enumerate}[{\normalfont (a)}]
		\item If $\zeta_{\ell} \notin K$, then the cyclic sub-extension $K\left(\zeta_{\ell}\right) $ of $K\left(P_{1}\right) $ over $K$ is non-trivial.
		\item If $\zeta_{\ell} \in K$, then the extension $K\left(P_{1}\right)/K$ is cyclic of degree $\ell$.
		\item For an integer $n\ge1$, if there exists a point $Q_{n+1} \in E\left[\ell^{n+1}\right]$ such that $E\left(K\right)\left[\ell^{\infty}\right] = \left\langle Q_{n} \right\rangle$, $Q_{1} = \ell^{n} Q_{n+1}$, and $Q_{n} =\ell Q_{n+1}$, then one of the following statements holds: \begin{enumerate}[{\normalfont (i)}]
			\item $K\left(Q_{n+1}\right) \cap K\left(\zeta_{\ell}\right)$ is a non-trivial cyclic sub-extension of $K\left(Q_{n+1}\right)$ over $K$.
			\item The extension $K\left(\zeta_{\ell}, Q_{n+1}\right) / K\left(\zeta_{\ell}\right)$ contains an abelian sub-extension of degree $\ell$ over $K\left(\zeta_{\ell}\right)$ and $\left[K\left(Q_{n+1}\right) : K\right]$ is either $\ell$ or $\ell^{2}$.
		\end{enumerate}
	\end{enumerate}
\end{prop}

\begin{proof}
	(a) It follows since $\zeta_{\ell} \in K(P_1)$, via the Weil pairing.
	
	(b)  Since $Q_{1} \in E\left(K\right)$ and $K\left(E\left[\ell\right]\right) = K\left(P_{1}\right)$, by Remark~\ref{remark:fixedQ}(a), we have $$
		H := r_{\left\{P_{1},Q_{1}\right\}}\left(\Gal{K\left(E\left[\ell\right]\right)/K}\right) \subseteq \left\{\begin{pmatrix}a&b\\0&1\end{pmatrix}\in \GLl\right\}.
	$$ Since $P_{1} \notin E\left(K\right)$, $H$ is non-trivial. If $\zeta_{\ell} \in K$, then since $Q_{1}$ is fixed under $\Gal{K\left(E\left[\ell\right]\right)/K}$ and $H \subseteq \SLl$, it follows from Remark~\ref{remark:fixedQ}(b) that $H= \left\langle U \right\rangle$ and the extension $K\left(P_{1}\right)/K$ is cyclic   of degree $\ell$.

	(c) Since $E\left(K\right)\left[\ell^{\infty}\right] = \left\langle Q_{n} \right\rangle$, by Lemma~\ref{lem:ell_sq}(b) with $k=0$ and $m=n$, we obtain $\left\langle P_{1},Q_{n+1}\right\rangle = \left\{T \in E\left[\ell^{\infty}\right]:\ell T \in E\left(K\right) \right\}$, and hence the extension $\Lt:= K\left(P_{1},Q_{n+1}\right)$ is Galois over $K$ by Lemma~\ref{lem:ell_sq}(a).

	If two extensions $K\left(Q_{n+1}\right)$ and $K\left(\zeta_{\ell}\right)$  are not linearly disjoint over $K$, then (i) holds. Now assume that $K\left(Q_{n+1}\right) \cap K\left(\zeta_{\ell}\right) = K$ and show (ii). Since $K\left(\zeta_{\ell}\right)$ is a Galois extension of $K$, we have \begin{align}\label{eqn:disj}
		\left[K\left(Q_{n+1}\right) : K\right] = \left[K\left(\zeta_{\ell}, Q_{n+1}\right) : K\left(\zeta_{\ell}\right)\right].
	\end{align}
	
	Let $L:=K\left(Q_{n+1}\right)$.
	If $Q_{n+1} \notin E\left(K\left(P_{1}\right)\right)$, then applying Lemma~\ref{lem:ell_sq}(c) to $K\left(P_{1},Q_{n+1}\right) / K\left(P_{1}\right)$ with $r=1$, we see that $\left[K\left(P_{1},Q_{n+1}\right) : K\left(P_{1}\right)\right]$ divides~$\ell^{2}$, which is the order of $\left(\bbz/\ell\bbz\right)^{2}$.
	Since $ \left[K\left(P_{1}\right) : K\left(\zeta_{\ell}\right)\right]=\ell$ by Part~(b) of this proposition,
	it follows that $$
		\left[\Lt : K\left(\zeta_{\ell}\right)\right] = \left[K\left(P_{1},Q_{n+1}\right) : K\left(P_{1}\right)\right] \left[K\left(P_{1}\right) : K\left(\zeta_{\ell}\right)\right] \text{ which divides } \ell^3.
	$$
	Thus, by \eqref{eqn:disj}, it follows that $\left[K\left(Q_{n+1}\right):K\right]$ divides $\ell^{3}$. We note that $\left[K\left(Q_{n+1}\right):K\right]>1$, since $E\left(K\right)\left[\ell^{\infty}\right] = \left\langle Q_{n} \right\rangle$.
	
	To show that $\left[K\left(Q_{n+1}\right):K\right]$ is either $\ell$ or $\ell^{2}$, it suffices to prove that this extension degree is less than or equal to $\ell^{2}$. The group $\Gal{\Lt/K}$ acts on the finite set $$
		S := \left\{X\in E\left[\ell^{n+1}\right]: \ell X =Q_{n}\right\}
	$$ of cardinality $\ell^2$,  and the degree $\left[K\left(Q_{n+1}\right):K\right]$ is equal to  the index of the stabilizer of $Q_{n+1}$ in $\Gal{\Lt/K}$, and this index coincides with the cardinality of the orbit $\left\{Q_{n+1}^{\sigma}:\sigma\in \Gal{\Lt/K}\right\}$, which is less than or equal to $\#S = \ell^{2}$.
	
	 Next, we show that $K\left(\zeta_{\ell}, Q_{n+1}\right) /  K\left(\zeta_{\ell}\right)$ contains a cyclic group of order $\ell$.

	Case 1:  Suppose $[\Lt : K\left(\zeta_{\ell}\right)] \mid \ell^{2}$. Since any group of order dividing $\ell^{2}$ is abelian, the sub-extension $K\left(\zeta_{\ell},Q_{n+1}\right)$ of $\Lt$ is also abelian of degree dividing $\ell^2$ over $K\left(\zeta_{\ell}\right)$. Since  $K\left(Q_{n+1}\right) \subseteq \Lt$ and two extensions $K\left(\zeta_{\ell}\right)$ and $K\left(Q_{n+1}\right)$ are linearly disjoint over $K$, it follows that $\ell \mid \left[K\left(\zeta_{\ell},Q_{n+1}\right) : K\left(\zeta_{\ell}\right)\right]$, and $K\left(\zeta_{\ell}, Q_{n+1}\right) /  K\left(\zeta_{\ell}\right)$ contains a cyclic sub-extension of order~$\ell$.

	Case 2: Suppose $ [\Lt : K\left(\zeta_{\ell}\right) ] = \ell^{3}$. Since $Q_{n}, Q_{1} \in E\left(K\right)$, there exist functions $$
		f, \; g,\; h : \Gal{\Lt/K\left(\zeta_{\ell}\right)} \rightarrow \bbz/\ell\bbz
	$$ such that for each $\sigma \in \Gal{\Lt/K\left(\zeta_{\ell}\right)}$, $$
		Q_{n+1}^{\sigma} - Q_{n+1} = f\left(\sigma\right) P_{1} + g\left(\sigma\right) Q_{1}, \quad \text { and }\quad  P_{1}^{\sigma} = P_{1} + h\left(\sigma\right) Q_{1},
	$$ which induces the function $$
		\phi := \begin{pmatrix} 1&f&g\\0&1&h\\0&0&1\end{pmatrix}: \Gal{\Lt/K\left(\zeta_{\ell}\right)} \to \left\{\begin{pmatrix} 1&a&b\\0&1&c\\0&0&1\end{pmatrix}: a,b,c \in \bbz/\ell\bbz \right\}
	$$  with respect to a choice of $\left\{Q_{n+1},P_{1},Q_{1}\right\}$. Then,  for any $\sigma, \tau \in \Gal{\Lt/K\left(\zeta_{\ell}\right)}$, we have that \begin{align*}
		f\left(\sigma\tau\right) P_{1} + g\left(\sigma\tau\right) Q_{1} &= Q_{n+1}^{\sigma \tau} - Q_{n+1}\\
		& = \left(Q_{n+1}^{\sigma}-Q_{n+1}\right)^{\tau} + \left(Q_{n+1}^{\tau}-Q_{n+1}\right) \\
		&= \left(f\left(\sigma\right) P_{1} + g\left(\sigma\right) Q_{1}\right)^{\tau} + \left(f\left(\tau\right) P_{1} + g\left(\tau\right) Q_{1}\right)\\
		&= \left(f\left(\sigma\right) P_{1} + \left(f\left(\sigma\right)h\left(\tau\right)+g\left(\sigma\right)\right) Q_{1}\right) + \left(f\left(\tau\right) P_{1} + g\left(\tau\right) Q_{1}\right)\\
		&= \left(f\left(\sigma\right)+f\left(\tau\right)\right) P_{1} + \left(f\left(\sigma\right)h\left(\tau\right)+g\left(\sigma\right) + g\left(\tau\right) \right)Q_{1},\end{align*}
		\begin{align*}
		\text{ and } \; h\left(\sigma \tau\right) Q_{1} & = P_{1}^{\sigma\tau} -P_{1}\\
		& = \left(P_{1}^{\sigma} -P_{1}\right)^{\tau} + \left(P_{1}^{\tau} -P_{1}\right)\hspace{4.7cm}\\
		& = \left(h\left(\sigma\right)+h\left(\tau\right)\right) Q_{1}.
	\end{align*}
	Therefore, both $\phi$ and $f$ are group homomorphisms. Clearly, $\Gal{\Lt/K\left(\zeta_{\ell}, Q_{n+1}\right)} \subseteq \ker f$. Hence, the fixed subfield $F$ under $\ker f$ is contained in $K\left(\zeta_{\ell}, Q_{n+1}\right)$, and  $\Gal{F/K\left(\zeta_{\ell}\right)} \cong \Gal{\Lt/K\left(\zeta_{\ell}\right)} / \ker f \cong \Im f$. Since $\left| \Im \phi \right| = \left[\Lt:K\left(\zeta_{\ell}\right)\right] = \ell^{3}$, the homomorphism $f$ can not be trivial, and thus, $\Gal{F/K\left(\zeta_{\ell}\right)} \cong  \Im f =\bbz/\ell\bbz$. Therefore, Case (ii) holds.
\end{proof}

 Combining Lemma~\ref{lem:ell_sq} with Proposition~\ref{prop:ell_linear}, we obtain a Galois-theoretic sufficient condition on $L/K$ and $\ell$ for the equality $E\left(L\right)\left[\ell^{\infty}\right] = E\left(K\right)\left[\ell^{\infty}\right]$, provided that  $E\left(K\right)\left[\ell\right] \ne \left\{O\right\}$, as follows:

\begin{cor}\label{cor:ell_non_zero}
	Let $K$ be a field of characteristic $0$ and $E/K$ an elliptic curve defined over~$K$. Suppose that $E\left(K\right)\left[\ell^{\infty}\right] \ne \left\{ O \right\}$,  and that for an extension $L/K$  we have $E\left(L\right)\left[\ell^{\infty}\right] \supsetneq E\left(K\right)\left[\ell^{\infty}\right]$. Then one of the following holds: \begin{enumerate}[{\normalfont (i)}]
		\item  The extension $L/K$ contains a non-trivial cyclic sub-extension.
		\item  The extension $L\left(\zeta_{\ell}\right)/K\left(\zeta_{\ell}\right)$ contains a cyclic sub-extension of degree $\ell$, and $L/K$ contains a sub-extension of degree $\ell$ or $\ell^{2}$.
	\end{enumerate}
\end{cor}
\begin{proof}
	Since $E\left(L\right)\left[\ell^{\infty}\right] \supsetneq E\left(K\right)\left[\ell^{\infty}\right]$, there exists a point $T \in E\left(L\right)\left[\ell^{\infty}\right] \setminus E\left(K\right)\left[\ell^{\infty}\right]$ such that $\ell T \in E\left(K\right)\left[\ell^{\infty}\right]$. Indeed, if not, for a point $R \in E\left(L\right)\left[\ell^{\infty}\right] \setminus E\left(K\right)\left[\ell^{\infty}\right]$, an induction would show that $\ell^{k} R \in E\left(L\right)\left[\ell^{\infty}\right] \setminus E\left(K\right)\left[\ell^{\infty}\right]$ for all $k\geq 1$, which is impossible. After possibly replacing $L$ by a smaller extension, we may therefore assume $L = K\left(T\right)$.
	
	If $E\left(K\right)\left[\ell\right] \cong \left(\bbz/\ell\bbz\right)^{2}$, then by Lemma~\ref{lem:ell_sq}(c) with $r=1$, the extension $K\left(T\right) / K$ has an abelian sub-extension of degree $\ell$, which proves (ii). If $E\left(K\right)\left[\ell\right] \cong \bbz/\ell\bbz$ and $\ell T =O$, then Proposition~\ref{prop:ell_linear}(a) and Proposition~\ref{prop:ell_linear}(b) show that $L/K$ contains a non-trivial cyclic extension, proving~(i). If $E\left(K\right)\left[\ell\right] \cong \bbz/\ell\bbz$ and $\ell T \ne O$, then $E\left(K\right)\left[\ell^{\infty}\right] = \left\langle \ell T\right\rangle$. Parts (i) and (ii) follow from Proposition~\ref{prop:ell_linear}(c)-(i) and Proposition~\ref{prop:ell_linear}(c)-(ii), respectively.
\end{proof}

\section{Proofs of Theorem~\ref{thm:thm2} and Theorem~\ref{thm:Sn}}\label{sec:additional}

We  now  prove Theorem~\ref{thm:thm2} by applying Lemma~\ref{lem:ell_sq} and performing induction on appropriately chosen torsion subgroups of an elliptic curve.

\begin{proof}[Proof of Theorem~\ref{thm:thm2}]
	If $E\left(K\right) \supseteq E\left[\ell^{s}\right]$ for all $s\ge0$, then the equality is trivial. Otherwise, there exists a maximal integer $r$ such that $E\left(K\right) \supseteq \left(\bbz/\ell^{r}\bbz\right)^{2}$. Since $\ell^{m} \ge 3$ and $E\left(K\right) \supseteq \left(\bbz/\ell^{m}\bbz\right)^{2}$, it follows that $\ell\ge 3$ or $r\ge 2$.

	For  $k\geq 0$, set $K_{k} := K\left(\scrT_{k}\right)$. We claim that for each integer $i\ge 0$, $E\left(K_{\left(2^{i}-1\right)r}\right)\left[\ell^{\infty}\right] = \scrT_{\left(2^{i}-1\right)r}$ for all $i$ by induction on $i$:  First, note that \begin{align}\label{subtor}
		E\left(K_{\left(2^{i}-1\right)r}\right)\left[\ell^{\infty}\right]\supseteq \scrT_{\left(2^{i}-1\right)r} \supseteq E\left[\ell^{2^{i}r}\right].
	\end{align} If $i=0$, then it follows directly from the definition of $\scrT_{0}$. Suppose $E\left(K_{\left(2^{i}-1\right)r}\right)\left[\ell^{\infty}\right] = \scrT_{\left(2^{i}-1\right)r}$ for some $i\geq 0$. Then, $$
		\scrT_{\left(2^{i+1}-1\right)r} = \left\{P \in E\left[\ell^{\infty}\right]:\ell^{2^{i}r} P \in \scrT_{\left(2^{i}-1\right)r} = E\left(K_{\left(2^{i}-1\right)r}\right)\left[\ell^{\infty}\right]\right\}=:\scrT_{2^{i}r}'
	$$ which provides an analogous definition of $\scrT_{k}$ over  the base field $K':=K_{\left(2^{i}-1\right)r}$, and $$
		K_{\left(2^{i+1}-1\right)r} = K_{\left(2^{i}-1\right)r}\left(\scrT_{\left(2^{i+1}-1\right)r}\right)=K'(\scrT_{2^{i}r}').
	$$ Therefore, since $\ell \ge 3$ or $2^{i}r \ge 2$, applying Lemma~\ref{lem:ell_sq}(e) over $K_{\left(2^{i}-1\right)r}$ together with considering~\eqref{subtor} yields $$
		E\left(K_{\left(2^{i+1}-1\right)r}\right) \left[\ell^{\infty}\right]
		= E\left(K'(\scrT_{2^{i}r}')\right) \left[\ell^{\infty}\right]
		=\scrT_{2^{i}r}'=\scrT_{\left(2^{i+1}-1\right)r},
	$$ which completes the induction.

	Now for a given $k\geq 0$, choose an integer $i_{0}\geq 0$ such that $\left(2^{i_{0}}-1\right)r \le k < \left(2^{i_{0}+1}-1\right)r$. Then, by the claim and \eqref{subtor}, we have $$
		E\left(K_{\left(2^{i_{0}}-1\right)r}\right) \left[\ell^{\infty}\right] = \scrT_{\left(2^{i_{0}}-1\right)r} \supseteq E\left[\ell^{2^{i_{0}}r}\right] \supseteq E\left[\ell^{k-\left(2^{i_{0}}-1\right)r+1}\right].
	$$ Since $\scrT_{k} = \left\{P \in E\left[\ell^{\infty}\right]:\ell^{k-\left(2^{i_{0}}-1\right)r} P \in \scrT_{\left(2^{i_{0}}-1\right)r} \right\}$ and $K_{k} = K_{\left(2^{i_{0}}-1\right)r}\left(\scrT_{k}\right)$, another application of Lemma~\ref{lem:ell_sq}(e) gives $E\left(K_{k}\right) \left[\ell^{\infty}\right] = \scrT_{k}$.
\end{proof}

To prove Theorem~\ref{thm:Sn}, we need Lemma~\ref{lem:L_cond} and the classification of the subgroups of $\GLl$.

\begin{lemma}\label{lem:L_cond}
	Let $K$ be a field of characteristic $0$, $L/K$  an extension, and $\Lt/K$ a Galois extension over $K$ containing $L$ with Galois group $G$.  Set $H := \Gal{\Lt/L}$. Then, $L/K$ has no non-trivial abelian sub-extension  over $K$ if and only if $G'H = G$, where $G'$ is the commutator subgroup of $G$.
\end{lemma}
\begin{proof}
	For a sub-extension $F/K$ of $\Lt/K$, $F/K$ is Galois and $\Gal{F/K}$ is abelian if and only if $\Gal{\Lt/F} \supseteq G'$, and $F\subseteq L$ if and only if $\Gal{\Lt/F} \supseteq H$. Hence, $K$ is the only one abelian sub-extension of $L/K$ if and only if the only subgroup of $G$ containing $G'H$ is $G$, i.e., $G=G'H$.
\end{proof}

\begin{lemma}[{\cite[Section~2]{Serre72}}]\label{lem:GL_sub}
	Let $\ell$ be a prime. For a subgroup $H$ of $\GLl$,  let $$
		\overline{H} := H / \left(H \cap \left\{aI_{2} :a\in \mathbb{F}_\ell^\times\right\}\right)
	$$ be the image of $H$ under the natural projection $\GLl \to \operatorname{PGL}_{2}\left(\Fl\right)$. Then one of the following holds:  
	\begin{enumerate}[{\normalfont (A)}]
		\item $H$ equals a group $\left\{A \in \GLl: \det A \in R\right\}$ for some subgroup $R \subseteq \Flx$, and in particular, $H$ contains $\mathrm{SL}_{2}(\mathbb{F}_{\ell})$;
		\item $H$ is a subgroup of a Borel subgroup of $\GLl$ and $\ell \mid \left|H\right|$, i.e., there exists $T \in \GLl$ such that $THT^{-1}$ is a subgroup of $\left\{\begin{pmatrix}a&b\\0&d\end{pmatrix}:a,d \in \Flx, b\in \Fl\right\}$ and $U \in THT^{-1}$;
		\item $H$ is a subgroup of the normalizer $\cN$ of a Cartan subgroup $\cC$ of $\GLl$;
		\item $\overline{H}$ is isomorphic to either $\Alt_{4}$, $\Sym_{4}$, or $\Alt_{5}$.
	\end{enumerate}
\end{lemma}

Now we prove Theorem~\ref{thm:Sn}.

\begin{proof}[Proof of Theorem~\ref{thm:Sn}]
	We let $\Lt/K$ be the Galois closure of $L/K$, i.e., the splitting field of~$f$.  Consider an isomorphism $\phi: \Gal{\Lt/K} \to \Sym_{n}$ such that $\alpha_{i}^{\sigma} = \alpha_{\phi\left(\sigma\right)\left(i\right)}$. Let $G := \phi\left(\Gal{\Lt/K}\right)=\Sym_{n}$ and $H := \phi\left(\Gal{\Lt/L}\right)$. Then, $H = \left\langle\left(1\; 2\right)\right\rangle$. Since the commutator subgroup $G'$ is equal to $\Alt_{n}$, Lemma~\ref{lem:L_cond}  implies: \begin{equation}\label{eqn:no_non_triv_ab_subextn_1}
		L/K \text{ has no non-trivial abelian sub-extension over } K.
	\end{equation}
	
	Let $\ell$ be a prime. First,  suppose $E\left(K\right)\left[\ell\right] \ne \left\{O\right\}$. 

	If $\ell=2$, then $L\left(\zeta_{\ell}\right) = L$ and $K\left(\zeta_{\ell}\right) = K$. In this case, Corollary~\ref{cor:ell_non_zero}~(ii) implies Corollary~\ref{cor:ell_non_zero}~(i). But  \eqref{eqn:no_non_triv_ab_subextn_1} shows that Corollary~\ref{cor:ell_non_zero}~(i) fails, hence so does Corollary~\ref{cor:ell_non_zero}~(ii). Therefore,  Corollary~\ref{cor:ell_non_zero} gives $$
		E\left(L\right)\left[\ell^{\infty}\right] = E\left(K\right)\left[\ell^{\infty}\right].
	$$

	Now, let $\ell\ge 3$. Since $$
		\Gal{\Lt\left(\zeta_{\ell}\right)/K\left(\zeta_{\ell}\right)} \cong \Gal{\Lt/\Lt \cap K\left(\zeta_{\ell}\right)}\subseteq \Sym_{n}
	$$ and $\Lt \cap K\left(\zeta_{\ell}\right)$ is an abelian extension over $K$, the group $\Gal{\Lt\left(\zeta_{\ell}\right)/K\left(\zeta_{\ell}\right)}$ is a non-trivial normal subgroup of $\Sym_{n}$, hence isomorphic to either $\Sym_{n}$ or $\Alt_{n}$. Consequently, any abelian sub-extension of $\Lt\left(\zeta_{\ell}\right)/K\left(\zeta_{\ell}\right)$ is trivial or quadratic and  the same holds for $L\left(\zeta_{\ell}\right)/K\left(\zeta_{\ell}\right)$, since $L\left(\zeta_{\ell}\right) \subseteq \Lt\left(\zeta_{\ell}\right)$. Thus, Corollary~\ref{cor:ell_non_zero}~(ii) fails, and since  Corollary~\ref{cor:ell_non_zero}~(i) also fails by \eqref{eqn:no_non_triv_ab_subextn_1}, Corollary~\ref{cor:ell_non_zero} again yields $$
		E\left(L\right)\left[\ell^{\infty}\right] = E\left(K\right)\left[\ell^{\infty}\right].
	$$

	Next, suppose $E\left(K\right)\left[\ell\right] = \left\{O\right\}$. Assume for the contradiction that $E\left(L\right)$ contains a point~$Q_{1}$ of order~$\ell$.
	First, we claim that the Galois closure of $K\left(Q_{1}\right)/K$ is $K\left(E\left[\ell\right]\right)$. To see this, it is enough to show the existence of $\sigma_{0} \in G_{K}$ such that $Q_{1}^{\sigma_{0}} \notin \left\langle Q_{1} \right\rangle$. If such a $\sigma_{0}$ does not exists, then $K\left(Q_{1}\right)/K$ is Galois and the function $f: \Gal{K\left(Q_{1}\right)/K} \to \left(\bbz/\ell\bbz\right)^{\times}$ defined by $Q_{1}^{\sigma} = f\left(\sigma\right) Q_{1}$ is an injective group homomorphism. Since $E\left(K\right)\left[\ell\right] = \left\{O\right\}$ and $\Gal{K\left(Q_{1}\right)/K} \cong \Im f \subseteq \left(\bbz/\ell\bbz\right)^{\times}$, $K\left(Q_{1}\right)/K$ is a non-trivial abelian sub-extension of $L/K$, contradicting \eqref{eqn:no_non_triv_ab_subextn_1}. Hence, the Galois closure of $K\left(Q_{1}\right)/K$ is $K\left(E\left[\ell\right]\right)/K$.

	Since $L \supseteq K\left(Q_{1}\right)$, we have $\Lt \supseteq K\left(E\left[\ell\right]\right)$. Put $N:= \phi\left(\Gal{\Lt/K\left(E\left[\ell\right]\right)}\right)$. Then $N$ is a normal subgroup of $G = \Sym_{n}$. For $n \ge 5$, $N$ is $\Sym_{n}$, $\Alt_{n}$, or the trivial subgroup. If $N= \Sym_{n}$, then $K = K\left(E\left[\ell\right]\right)$, contradicting  $E\left(K\right)\left[\ell\right] = \left\{O\right\}$. If $N= \Alt_{n}$, then $\Gal{K\left(E\left[\ell\right]\right)/K} = \Sym_{n}/\Alt_{n} \cong \bbz/2\bbz$. Since $E\left(K\right)\left[\ell\right] = \left\{O\right\}$, it follows that $K\left(Q_{1}\right)=K\left(E\left[\ell\right]\right)$ would give a non-trivial abelian subextension of $L/K$,  contradicting~\eqref{eqn:no_non_triv_ab_subextn_1} in the above. If $N$ is the trivial subgroup, then $\Gal{K\left(E\left[\ell\right]\right)/K} =G/N \cong \Sym_{n}$. By Lemma~\ref{lem:GL_sub}, $\Sym_{n}$ would fit into one of the following structures inside $\GLl$:
	\begin{enumerate}[{\normalfont (A)}]
		\item $\Sym_{n}$ contains a normal subgroup which is isomorphic to $\SLl$;
		\item $\Sym_{n}$ is isomorphic to a subgroup of a Borel subgroup of $\GLl$ and $\left|\Sym_{n}\right|$ is divisible by $\ell$;
		\item $\Sym_{n}$ is isomorphic to a subgroup $H$ of the normalizer $\cN$ of a Cartan subgroup $\cC$ of $\GLl$;
		\item $\Sym_{n}$ contains a cyclic normal subgroup $C$  with quotient $\Sym_{n}/C$   isomorphic to either $\Alt_{4}$, $\Sym_{4}$, or $\Alt_{5}$.
	\end{enumerate}
	Showing that each of (A)~$\sim$~(D) cannot happen, would prove that $E\left(L\right)\left[\ell\right] = \left\{O\right\}$.
	
	If $\ell \in \left\{ 2,3 \right\}$, a size comparison of  the orders of $\Sym_{n}$ and $\GLl$ immediately rules out embeddings $\Sym_{n} \cong \Gal{K\left(E\left[\ell\right]\right)/K} \subseteq \GLl$. So let $\ell \ge 5$.
		
	First, suppose (A) happens. Then since the commutator subgroups satisfy that
	$\left(\SLl\right)'=\SLl=\left(\GLl\right)'$ by \cite[Theorem~1.9]{Gr}, we conclude that $\Alt_{n} \cong \SLl$. Since  $\ell \mid \left|\SLl\right| = \left|\Alt_{n}\right|$, we have $\ell \le n$. But if $n\ge \ell+1$, this contradicts the following since $\ell\geq 5$: $$
		1= \frac{\left|\Alt_{n}\right|}{\left|\SLl\right|} = \frac{\left(\ell-2\right)!}{2}\frac{n!}{\left(\ell+1\right)!}.
	$$ If $n=\ell$, we have $\frac{\left(\ell-2\right)!}{2\left(\ell+1\right)} =\frac{\left|\Alt_{\ell}\right|}{\left|\SLl\right|} = 1$. However, $\frac{\left(5-2\right)!}{2\left(5+1\right)} \ne 1$ for $\ell=5$ and it is easy to show that $\frac{\left(k-2\right)!}{2\left(k+1\right)} > 1$ for all $k\ge6$ by induction.

	Case (B) is impossible, since a Borel subgroup of $\GLl$ has a normal subgroup of order $\ell$ by Lemma~\ref{lem:Borel}, while the order of a normal subgroup of $\Sym_{n}$ for $\ell\geq 5$ is either $1$, $\frac{n!}{2}$, or $n!$.
	
	Suppose (C) happens. First, we note that $\left[ \cN: \cC \right] = 2$ by \cite[Chapter~XVIII.Proposition~12.3]{Lang}. Therefore, $\cC H$ is either $\cN$ or $\cC$. Since $\cC$ is abelian but $H$ is not abelian, it follows that  $\cC H = \cN$. Hence, \begin{align}\label{eq:intersecIndex}
		\left[H:H\cap\cC\right] = \left[\cC H: \cC \right] = \left[\cN:\cC\right] = 2,
	\end{align} which implies that  $H\cap\cC$ is a normal subgroup of index~$2$ in $H\cong \Sym_{n}$, so $H\cap\cC \cong \Alt_{n}$, which is a contradiction since $\cC$ is abelian.
	
	Case (D) is impossible, since the unique cyclic normal subgroup $C$ of $\Sym_{n}$ is the trivial subgroup for  $n\geq 5$.
\end{proof}

\

\section{Over non-Galois extensions of odd prime degree}\label{sec:zero}

In this section, we consider a non-Galois extension $L/K$   of prime degree $p$ and provide sufficient conditions on $K$ and a prime $\ell$ under which the $\ell$-primary torsion of $E(L)$ coincides with that of $E(K)$. When $E\left(K\right)\left[\ell\right] \ne \left\{O\right\}$, Corollary~\ref{cor:ell_non_zero} shows that this equality holds for $\ell \ne p$ almost immediately. Therefore, the main task of this section is to deal with the case $E\left(K\right)\left[\ell\right] = \left\{O\right\}$, which is treated separately  in the proofs of Proposition~\ref{prop:non-Gal_3} and Proposition~\ref{prop:non-Gal_prime}.

\begin{prop}\label{prop:non-Gal_3}
	Let $K$ be a field of characteristic $0$ and $E/K$ an elliptic curve  over $K$ whose affine model is defined by $y^2=f\left(x\right)$, where $f\left(x\right)\in K\left[x\right]$ is a monic cubic polynomial. Let $L/K$ be a non-Galois cubic extension and $\ell$ a prime such that $E\left(K\right)\left[\ell\right]= \left\{O\right\}$. Then:
	\begin{enumerate}[{\normalfont (a)}]
		\item
		If $\ell =2$ and $L$ has no zero of $f$, then $E\left(L\right)\left[2\right] =\left\{O\right\}= E\left(K\right)\left[2\right]$.
		\item 
		If $\ell\ge 3$ and $\Im \cyc \neq \left\{\pm 1\right\}$, then $E\left(L\right)\left[\ell\right] =\left\{O\right\}= E\left(K\right)\left[\ell\right]$.
	\end{enumerate}
\end{prop}
\begin{proof}
	First, note that since $L/K$ is a non-Galois cubic extension, for any element  $\alpha \in L \setminus K$, the minimal polynomial $g_\alpha$ of $\alpha$ over $K$ is cubic, and we have  $\Gal{g_\alpha/K} \cong \Sym_{3}$.

	Part~(a) is immediate, since if $E\left(L\right)\left[2\right] \supsetneq E\left(K\right)\left[2\right]$ then $L$ contains a zero of $f$.
	
	To prove (b),  for a prime  $\ell\geq 3$ such that $E\left(K\right)\left[\ell\right] = \left\{O\right\}$, suppose  $E\left(L\right)$ has a point  $Q_{1}$ of order~$\ell$ and show that $\Im \cyc = \left\{\pm 1\right\}$. Because $\left[L:K\right]$ is a prime, we have $L=K\left(Q_{1}\right)$. Following an analogous argument as in the proof of Theorem~\ref{thm:Sn}, the Galois closure of $L/K$ is $K\left(E\left[\ell\right]\right)$. Since $L/K$ is a non-Galois cubic extension, we have that $\Gal{K\left(E\left[\ell\right]\right)/K}\cong \Sym_3$ and, by Lemma~\ref{lem:GL_sub}, there are four possible cases for how $\Sym_3$ can appear in $\GLl$ as follows: \begin{enumerate}[{\normalfont (A)}]
		\item $\Sym_{3}$ contains a normal subgroup which is isomorphic to $\SLl$;
		\item $\Sym_{3}$ is isomorphic to a subgroup of a Borel subgroup of $\GLl$ with $\ell$ dividing $\left|\Sym_{3}\right|=6$;
		\item For a basis $\cB$ of $E\left[\ell\right]$, $H := r_{\cB}\left(\Gal{K\left(E\left[\ell\right]\right)/K}\right)$ is isomorphic to $\Sym_{3}$ as a subgroup of the normalizer $\cN$ of a Cartan subgroup $\cC$ of $\GLl$;
		\item $\Sym_{3}$ contains a cyclic normal subgroup $C$ with the quotient $\Sym_{3}/C$ isomorphic to either $\Alt_{4}$, $\Sym_{4}$, or $\Alt_{5}$.
	\end{enumerate}
	
	Case (A) is impossible by comparing the orders of $\Sym_{3}$ and $\SLl$ for $\ell\geq 3$. Similarly, Case (D) is impossible  by comparing the order of $\Alt_{4}$, $\Sym_{4}$, or $\Alt_{5}$, with the order of $\Sym_{3}$, respectively.
	
	Suppose (B) happens. Then $\ell=3$, since $3 \le \ell \mid \left|\Sym_{3}\right|=6$. Then, considering Lemma~\ref{lem:GL_sub}, there exists a basis $\cB := \left\{P_{1}, Q_{1}\right\}$ of $E\left[3\right]$ such that $$
		G := r_{\cB}\left(\Gal{K\left(E\left[\ell\right]\right)/K}\right) \subseteq \left\{\begin{pmatrix}a&b\\0&d\end{pmatrix}\in \operatorname{GL}_{2}\left(\bbf_{3}\right)\right\}
	$$ and $U \in G$. Then, as in~\eqref{eqn:gen} in the proof of Lemma~\ref{lem:Borel}, $G$ is generated by $U$ and $\Delta := \left\{\diag{a,d}\in G: a,d\in\bbf_{3}^{\times}\right\}$. Since $G\cong \Sym_3$, $\Delta$ has order $2$. If $\Delta = \left\langle \diag{-1,-1} \right\rangle$, $G$ is a cyclic group of order $6$, contradicting that $G \cong \Sym_{3}$ is non-abelian. Hence, this implies that $\det \Delta = \left\{\pm1\right\}$, and by Remark~\ref{remark:fixedQ}(b), we have $\Im \operatorname{cyc}_{3} =\det\left(\Im r_{\cB}\right)= \det G = \det \Delta = \left\{\pm1\right\}$.
	
	Suppose (C) happens. As in \eqref{eq:intersecIndex} of the proof of Theorem~\ref{thm:Sn},  we have $\left[H : H \cap \cC\right] = 2$, and so $H\cap\cC$ is generated by a matrix $R$ of order $3$. Considering the group structure of $H \cong \Sym_{3}$, for $F \in H \setminus \cC$,  we have that $F^{2} = I_{2}$ and $FRF^{-1} = R^{-1}$. Hence, $\left(\det R\right)^{2} = \left(\det \left(FRF^{-1}\right)\right) \left(\det R\right) = 1$ and so $\det R \in \left\{\pm1\right\}$. Then, it implies $\det R=1$ since $R^{3} = I_{2}$, and that $H \cap \cC$ is a subgroup of the Cartan subgroup $\cC$, we have that $\ell \ne 3$ since the order of $\cC$ when $\ell=3$ is either $4$ (split Cartan) or $8$ (non-split Cartan),\ and  so $R$ of order $3$ has two distinct eigenvalues $\omega$ and $\omega^{-1}$ in $\Fll$ where $\omega$ is a primitive cube root of unity. Thus, there exists $T \in \GLll$ such that $T R T^{-1} = \diag{\omega,\omega^{-1}}$. For $F_{*} := T F T^{-1}$, since the equations $F_{*} \diag{\omega,\omega^{-1}} F_{*}^{-1} = \diag{\omega^{-1},\omega}$, and  $F_{*}^{2} = I_{2}$ are satisfied, it follows that there exists $b\in \Fllx$ such that $F_{*} = \begin{pmatrix}0&b\\1/b&0\end{pmatrix}$. Therefore, by Remark~\ref{remark:fixedQ}(b) and the definition of $H$, $\det F = \det F_{*} = -1$ and $\Im \cyc = \det H = \det \left\langle R, F\right\rangle = \left\{\pm1\right\}$.
\end{proof}

\begin{prop}\label{prop:non-Gal_prime}
	Let $K$ be a field of characteristic $0$ and $E/K$ an elliptic curve  over $K$. Let $L/K$ be a non-Galois extension of prime degree $p\ge 5$. For a prime $\ell \ne p$ such that $E\left(K\right)\left[\ell\right] = \left\{O\right\}$, if $\Im \cyc \neq \left\{\pm 1\right\}$, then $E\left(L\right)\left[\ell\right] = \left\{O\right\}= E\left(K\right)\left[\ell\right]$.
\end{prop}

\begin{proof}

	Suppose that there exists $Q_{1}\in E\left(L\right)$ of order $\ell$ and we will show that  $\Im \cyc =\left\{\pm 1\right\}$. Then, $L = K\left(Q_{1}\right)$ as  $\left[L:K\right]=p$. Following a similar argument as in the proof of Theorem~\ref{thm:Sn}, the Galois closure $\Lt$ of $L$ over $K$ equals $K\left(E\left[\ell\right]\right)$. Fix a basis $\cB := \left\{P_{1}, Q_{1}\right\}$ of $E\left[\ell\right]$ and let $$
		G := r_{\cB}\left(K\left(E\left[\ell\right]\right)/K\right) \quad \text{ and } \quad H := r_{\cB}\left(K\left(E\left[\ell\right]\right)/L\right).
	$$ Since $\Lt/K$ is the Galois closure of the non-Galois extension $L/K$, it follows that  and that $H$ is not a normal subgroup of $G$. By the primitive element theorem (\cite[Ch.14~Theorem~25]{DF}), $L$ is generated by an element $\alpha \in L$ over $K$. The minimal polynomial $g_{\alpha}$ of $\alpha$ over $K$ has degree $p = \left[K\left(\alpha\right):K\right] = \left[L:K\right]$ and the Galois closure $\Lt/K$ of $L = K\left(\alpha\right)/K$ is the splitting field of $g_{\alpha}$ over $K$. Hence, there exists an injective group homomorphism $\phi: G \to \Sym_{p}$. By Lemma~\ref{lem:GL_sub}, one of the following holds: \begin{enumerate}[{\normalfont (A)}]
		\item $G = \left\{A \in \GLl: \det A \in R\right\}$ for a subgroup $R \subseteq \Flx$;
		\item $G$ is a subgroup of a Borel subgroup of $\GLl$ and $\ell \mid \left|G\right|$;
		\item $G$ is a subgroup of the normalizer $\cN$ of a Cartan subgroup $\cC$ of $\GLl$;
		\item $G$ contains a cyclic normal subgroup $N$ with quotient $G/N$ isomorphic to either $\Alt_{4}$, $\Sym_{4}$, or $\Alt_{5}$.
	\end{enumerate}
	First, we rule out (A), (B), and (D):

	Suppose (A) happens. Then by the definition of $r_{\cB}$ and Remark~\ref{remark:fixedQ}(a), we have $H = \left\{\begin{pmatrix}a&b\\0&1\end{pmatrix} \in \GLl: a \in R\right\}$ and $p = \left[G:H\right] = \frac{\left|R\right|\ell\left(\ell^{2}-1\right)}{\left|R\right|\ell} = \ell^{2}-1$, which is impossible for prime $p\ge 5$.

	Suppose (B) happens. Then, $\left|G\right|$ divides $\left(\ell-1\right)^{2} \ell$, and Lemma~\ref{lem:Borel} shows that every subgroup of $G$ of order  divisible by $\ell$ is  Galois in $G$. Since $\left[G:H\right] = p$ with $H$ non-normal, this forces $p=\ell$, contradicting the assumption $\ell \ne p$.

	Suppose (D) holds. Assume first that $p\ge 7$ or $G/N \not\cong \Alt_{5}$. Then, $p$ divides $\left|G\right|$ but not $\left[G:N\right]$, so $p \mid \left|N\right|$. Since $N$ is cyclic, let $M\subseteq N$ be the unique subgroup of order~$p$. For any $g \in G$, $gMg^{-1}$ is also a cyclic subgroup of order $p$ in $N = gNg^{-1}$, so  $gMg^{-1}=M$. Thus, $M$ is normal in $G$. Since the number of Sylow $p$-subgroups in $\Sym_{p}$ is $\left(p-1\right)!$ and the normalizer of $\left\langle \left( 1\,2\,\cdots\,p\right)\right\rangle$ in $\Sym_{p}$ is $\left\langle \left( 1\,2\,\cdots\,p\right)\right\rangle$ itself, via the embedding $\phi$ of $G$ into $\Sym_{p}$, we conclude that $G=M$ is cyclic, contradicting the non-normality of $H$ in $G$.

	Now suppose that $p=5$ and $G/N \cong \Alt_{5}$. As noted above, $G$ embeds into $\Sym_{5}$. Hence, the order of $G$ is either $5!$ or $\frac{5!}{2}$ and $G$ is isomorphic to either $\Sym_{5}$ or $\Alt_{5}$. In either case, we reach a contradiction by showing that $G$ contains a normal subgroup of order~$2$. Indeed, $G$ contains a subgroup $H$ which is isomorphic to $\bbz/2\bbz \times \bbz/2\bbz$. Since $H$ is abelian, there exists $T \in \GLll$ such that $THT^{-1}$ consists of diagonal matrices. Since $H \cong \bbz/2\bbz \times \bbz/2\bbz$, we have $$
		THT^{-1}=\left\{\diag{\pm1,\pm1}\right\},
	$$ and $-I_{2} \in H$. Therefore, $G$ contains a normal subgroup of order~$2$.

	Thus, case (C) remains.
	Suppose (C) holds. Since $H$ is not normal in $G$, $G$ is not abelian. Following the same argument as in \eqref{eq:intersecIndex} of the proof of Theorem~\ref{thm:Sn}, we have  $\left[G:G\cap\cC\right] = 2$. Since $p = \left[G:H\right]$ is relatively prime to $2$, $G$ is generated by $H$ and $M := G\cap\cC$. Since $\ell \nmid \left|\cN\right|$ and $H \subseteq G \subseteq \cN$, it follows that $\ell \nmid \left|H\right|$, and moreover, since $L=  K\left(Q_{1}\right)$, the group homomorphism $H \to \Flx$ defined by $\begin{pmatrix}a&b\\0&1\end{pmatrix} \mapsto a$ is injective, and thus, $H$ is a cyclic group. Let $\begin{pmatrix}a_{0}&b_{0}\\0&1\end{pmatrix}$ be a generator of $H$. Replacing the basis $\cB = \left\{P_{1},Q_{1}\right\}$ by $\left\{P_{1}+\frac{b_{0}}{a_{0}-1}Q_{1},Q_{1}\right\}$ (here, $a_{0} \ne 1$ since $\ell \nmid \left|H\right|$ and $H$ is non-trivial), Remark~\ref{remark:fixedQ}(c) implies that $H$ is generated by a single diagonal matrix $\diag{a_{0},1}$. Since $$
		2 = \left[G:M\right] = \left[MH:M\right] = \left[H:H \cap M\right] \mid \left|H\right|,
	$$ the order of $a_{0} \in \Flx$ is even and $H \cap M$ is generated by $\diag{a_{0}^{2},1}$. If $a_{0}^{2} \ne 1$, then $M$ consists of diagonal matrices since $M \subseteq \cC$ is abelian, which forces $G= MH$ to be abelian, contradicting the existence of $H$ not a normal subgroup of $G$. Therefore, $a_{0}^{2} = 1$. Again, since $H$ is not a normal subgroup of $G$, it follows that $a_{0}=-1$ and $H \cap M = \left\{I_{2}\right\}$, thus, $|H|=2$. In other words, $G$ is a semi-direct product of $H$ acting on $M$. A matrix $A \in M$ of order $p$ generates $M$ because $\left|M\right| = \left[G:H\right] = p$. Because $H$ is not a normal subgroup of $G$, the action of $H$ on $M$ which defines the semi-direct product $G$ is not trivial. Since $H \cong \bbz/2\bbz$, the unique non-trivial action satisfies $\diag{-1,1} \;A \;\diag{-1,1}^{-1} = A^{-1}$. Hence, $\left(\det A\right)^{2} = \left(\det \left(\diag{-1,1}A\diag{-1,1}^{-1}\right)\right) \left(\det A\right) = \left(\det A^{-1}\right) \left(\det A\right)=1$, and so $\det A \in \left\{\pm1\right\}$. Then, since $A^{p}=I_{2}$, we have $\det A=1$ and $\Im \cyc = \det G = \left\langle \det A, \diag{-1,1}\right\rangle = \left\{\pm1\right\}$ by Remark~\ref{remark:fixedQ}(b).
\end{proof}

Now we obtain Theorem~\ref{thm:non-Gal_p} as follows:

\begin{proof}[Proof of Theorem~\ref{thm:non-Gal_p}]
	Part~(a) follows from Corollary~\ref{cor:ell_non_zero} for $p\neq \ell$, and from the argument analogous to one in the proof of Theorem~\ref{thm:Sn} for $p=3$.
	
	Part~(b) follows from Proposition~\ref{prop:non-Gal_3}(b) if $p=3$ and from  Proposition~\ref{prop:non-Gal_prime} for $\ell\neq p$ if $p\neq 3$.

	For Part~(c), suppose that $E\left(L\right)\left[2\right] \supsetneq \left\{O\right\}= E\left(K\right)\left[2\right]$. We aim to show $p=3$. $L$ contains a zero $\alpha$ of $f$ which does not lie in $K$. Since $f$ has a cubic polynomial with no zeros in $K$, it is irreducible over $K$,  and thus $\left[K\left(\alpha\right):K\right] = 3$. As $p = \left[L:K\right]$ is   prime and $L \supseteq K\left(\alpha\right)$, it follows that $3 = \left[K\left(\alpha\right):K\right] = \left[L:K\right] = p$.
\end{proof}

\section{Applications of Theorem~\ref{thm:non-Gal_p}}\label{sec:application}

As an application of Theorem~\ref{thm:non-Gal_p}, we obtain  Corollary~\ref{cor:ref_Go17} and Corollary~\ref{cor:ref_Na16}, which refines \cite[Theorem~1]{Go17} and \cite[Theorem~1]{Na16} as follows:

\begin{proof}[Proof of Corollary~\ref{cor:ref_Go17}]
	\cite[Theorem~1]{Go17} classifies all torsion structures that can occur over quintic extensions of $\bbq$. To specialize to the non-Galois quintic case, it remains to rule out the groups $\bbz/11\bbz$ and $\bbz/25\bbz$ from that list when $L/\bbq$ is non-Galois.

	If $E\left(L\right)_{\tors} \cong \bbz/11\bbz$, then by Mazur's theorem~\cite{Mazur}, ${E\left(\bbq\right)}_{\tors} = \left\{O\right\}$. Theorem~\ref{thm:non-Gal_p}(b) shows that this is impossible in a non-Galois quintic extension, since $\Im\operatorname{cyc}_{13} = \left(\bbz/13\bbz\right)^{\times}$.

	If $E\left(L\right)_{\tors} \cong \bbz/25\bbz$, then by Mazur's theorem~\cite{Mazur}, ${E\left(\bbq\right)}_{\tors}$ is isomorphic to either the trivial group or $\bbz/5\bbz$. However, (b) and (a) in Theorem~\ref{thm:non-Gal_p} prove that ${E\left(\bbq\right)}_{\tors}$ can not be isomorphic to either the trivial group or $\bbz/5\bbz$.
\end{proof}

\begin{proof}[Proof of Corollary~\ref{cor:ref_Na16}]
	\cite[Theorem~1]{Na16} classifies all torsion structures that can occur over cubic extensions of $\bbq$. To specialize to the non-Galois cubic case, it remains to rule out the groups $\bbz/13\bbz$, $\bbz/21\bbz$, and $\bbz/2\bbz \oplus \bbz/14\bbz$ from that list when $L/\bbq$ is non-Galois. 

	If $E\left(L\right)_{\tors} \cong \bbz/13\bbz$, then by Mazur's theorem~\cite{Mazur}, ${E\left(\bbq\right)}_{\tors} = \left\{O\right\}$. Theorem~\ref{thm:non-Gal_p}(b) shows that this is impossible in a non-Galois cubic extension, since $\Im\operatorname{cyc}_{13} = \left(\bbz/13\bbz\right)^{\times}$.

	In \cite[Theorem~1]{Na16}, $E\left(L\right)_{\tors} \cong \bbz/21\bbz$ only for the elliptic curve $E$ labeled by 162b1 in \cite{LMFDB} and the unique cubic sub-extension $L$ of $\bbq\left(\zeta_{9}\right)/\bbq$. Therefore, this case is ruled out under our assumption on non-Galois $L/\bbq$.

	Suppose $E\left(L\right)_{\tors} \cong \bbz/2\bbz \oplus \bbz/14\bbz$. If $E\left(\bbq\right)_{\tors}$ has order not divisible by~$7$, then since $\Im\operatorname{cyc}_{7} = \left(\bbz/7\bbz\right)^{\times}$, Theorem~\ref{thm:non-Gal_p}(b) implies that $E\left(L\right)_{\tors}$ also cannot have order divisible by~$7$, which is a contradiction. Therefore, we must have $E\left(\bbq\right)_{\tors} \cong \bbz/7\bbz$ by Mazur's theorem~\cite{Mazur}. Writing $E/\bbq$ by $y^{2} = f\left(x\right)$, where  $f \in \bbq\left[x\right]$ is a cubic polynomial, $f$ has no $\bbq$-rational zero but splits completely over $L$. Since $\left[L:\bbq\right]=3$,  $L$ is necessarily the splitting field of $f$, hence Galois, leading a contradiction.
\end{proof}

In Theorem~\ref{thm:non-Gal_p}, for a non-Galois extension $L/K$ of prime degree $p$ and a prime $\ell \ne p$, if $E\left(K\right)\left[\ell\right] = \left\{O\right\}$, then a sufficient condition for $E\left(L\right)\left[\ell^{\infty}\right] = E\left(K\right)\left[\ell^{\infty}\right]$ is expressed  in terms of $\Im \cyc$. By restricting $\Im \cyc$, we obtain Corollary~\ref{cor:overQ_and_Qmu} as a further application of Theorem~\ref{thm:non-Gal_p}, as follows:

\begin{proof}[Proof of Corollary~\ref{cor:overQ_and_Qmu}]
	
	First, Let $K=\bbq$. If $E\left(\bbq\right)\left[\ell^{\infty}\right] \ne \left\{O\right\}$, then $E\left(L\right)\left[\ell^{\infty}\right] = E\left(\bbq\right)\left[\ell^{\infty}\right]$, so Theorem~\ref{thm:non-Gal_p}(a) completes the proof. On the other hand, if $E\left(\bbq\right)\left[\ell^{\infty}\right] = \left\{O\right\}$, then since  $\Im \cyc = \left(\bbz/\ell\bbz\right)^{\times}$, which is different from $\left\{\pm1\right\}$ when $\ell \ge 5$, Theorem~\ref{thm:non-Gal_p}(c) completes the proof.
	
	If $K=\bbq\left(\mu_{\infty}\right)$, the reasoning is the same as in the previous case, noting that $\Im \cyc = \left\{1\right\}$, which differs from $\left\{\pm1\right\}$ when $\ell \ge 3$.
    
\end{proof}

\noindent{\bf Acknowledgement.} The authors thank Daeyeol Jeon for pointing out the relevant reference~\cite{JS20}.


\begin{thebibliography}{9}



	\bibitem{BrNa17} \textit{Peter Bruin} and \textit{Filip Najman}
	Fields of definition of elliptic curves with prescribed torsion,
	Acta Arith. \textbf{181} (2017), no. 1, 85–95.
	
	
	\bibitem{Ch16} \textit{Michael Chou},
	Torsion of rational elliptic curves over quartic Galois number fields,
	J. Number Theory \textbf{160} (2016) 603--628.

	\bibitem{Ch19} \textit{Michael Chou},
	Torsion of rational elliptic curves over the maximal abelian extension of $\bbq$
	Pacific J. Math. \textbf{302} (2019), no. 2, 481--509.
	
	\bibitem{ChDKNa21} \textit{Michael Chou}, \textit{Harris B. Daniels}, \textit{Ivan Krijan}, and \textit{Filip Najman},
	Torsion groups of elliptic curves over the $\bbz_{p}$-extensions of $\bbq$,
	New York J. Math. \textbf{27} (2021), 99--123.
	
	\bibitem{DS17} \textit{Maarten Derickx} and \textit{Andrew V. Sutherland},
	Torsion subgroups of elliptic curves over quintic and sextic number fields.
	Proc. Amer. Math. Soc. \textbf{145} (2017), no. 10, 4233--4245.
	
	
	\bibitem{DF} \textit{David S. Dummit} and \textit{Richard M. Foote},
	Abstract algebra, 3rd ed.
	(John Wiley \& Sons, Inc., Hoboken, NJ 2004).

	\bibitem{Go17} \textit{Enrique Gon{\'a}lez-Jim{\'e}nez},
	Complete classification of the torsion structures of rational elliptic curves over quintic number fields,
	J. Algebra \textbf{478} (2017) 484--505.
	
	\bibitem{GoNa20} \textit{Enrique Gon{\'a}lez-Jim{\'e}nez} and \textit{Filip Najman},
	Growth of torsion groups of elliptic curves upon base change,
	Math. Comp. \textbf{89} (2020), no. 323, 1457-1485.
	
	\bibitem{GoTo14} \textit{Enrique Gonz{\'a}lez-Jim{\'e}nez} and \textit{Jos{\'e} M. Tornero},
	Torsion of rational elliptic curves over quadratic fields,
	Rev. R. Acad. Cienc. Exactas F{\'i}s. Nat. Ser. A Mat. RACSAM\textbf{108}(2014), no.2, 923--934.

	\bibitem{G24} \textit{Tyler Genao},
	Growth of torsion groups of elliptic curves upon base change from number fields.
	Ramanujan J. \textbf{63} (2024), no. 2, 409--429.

	\bibitem{Gr} \textit{Larry C. Grove},
	Classical groups and geometric algebra.
	Grad. Stud. Math., \textbf{39} American Mathematical Society, Providence, RI, 2002, x+169 pp.
	ISBN: 0-8218-2019-2

	\bibitem{IK22} \textit{Bo-Hae Im} and \textit{Hansol Kim},
	Density of elliptic curves over number fields with prescribed torsion subgroups.
	preprint 2022, \url{https://arxiv.org/abs/2209.02889v3}
	
	\bibitem{IK24} \textit{Bo-Hae Im} and \textit{Hansol Kim},
	Growth of torsion groups of elliptic curves over number fields without rationally defined CM.
	J. Number Theory \textbf{258} (2024), 1–21.

	\bibitem{JKS04} \textit{Daeyeol Jeon}, \textit{Chang Heon Kim}, and \textit{Andreas Schweizer},
	On the torsion of elliptic curves over cubic number fields,
	Acta Arith. \textbf{113} (2004), no. 3, 291--301. 
	
	\bibitem{JKP06} \textit{Daeyeol Jeon}, \textit{Chang Heon Kim}, and \textit{Euisung Park},
	On the torsion of elliptic curves over quartic number fields,
	J. London Math. Soc. (2) \textbf{74} (2006), no. 1, 1--12. 

	\bibitem{JS20} \textit{Daeyeol Jeon} and \textit{Andreas Schweizer},
	Torsion of rational elliptic curves over different types of cubic fields.
	Int. J. Number Theory \textbf{16} (2020), no. 6, 1307--1323.
	
	\bibitem{Kamienny} \textit{Sheldon Kamienny},
	Torsion points on elliptic curves and $q$-coefficients of modular forms,
	Invent. Math. \textbf{109} (1992) 221--229.
	
	\bibitem{Kenku_Momose} \textit{M. A. Kenku} and \textit{Fumiyuki Momose},
	Torsion points on elliptic curves defined over quadratic fields.
	Nagoya Math. J. \textbf{109} (1988), 125--149.

	\bibitem{Lang} \textit{Serge Lang},
	Algebra, Grad. Texts in Math.,
	211 Springer-Verlag, New York, 2002, xvi+914 pp. ISBN: 0-387-95385-X
	
	
	\bibitem{LMFDB} The LMFDB Collaboration, \textit{The $L$-functions and modular forms database}, {\it Varieties: Elliptic curves over~$\bbq(\alpha)$}, http://www.lmfdb.org, 2025, [Online; accessed 27 October 2025].

	\bibitem{Mazur} \textit{Barry Mazur},
	Modular curves and the Eisenstein ideal. With an appendix by Mazur and M. Rapoport.
	Inst. Hautes {\' E}tudes Sci. Publ. Math. No. \textbf{47} (1977), 33--186 (1978).

	\bibitem{Milne} \textit{James S. Milne},
	Fields and Galois Theory (v5.10).
	Available at \url{https://www.jmilne.org/math/}, 2022, 144 pp. (1978).

	\bibitem{Na16} \textit{Filip Najman},
	Torsion of rational elliptic curves over cubic fields and sporadic points on $X_{1}\left(n\right)$,
	Math. Res. Lett. \textbf{23}(1) (2016) 245--272.
	
	
	
	\bibitem{Serre72} \textit{Jean-Pierre Serre},
	Propri{\'e}t{\'e}s galoisiennes des points d’ordre fini des courbes elliptiques,
	Invent. Math \textbf{15} (1972), 259 – 331.
	
	\bibitem{Serre97} \textit{Jean-Pierre Serre},
	Lectures on the Mordell-Weil theorem. Translated from the French and edited by Martin Brown from notes by Michel Waldschmidt. With a foreword by Brown and Serre. 3rd ed.
	Vieweg \& Sohn, Braunschweig 1997.
	
	\bibitem{Serre08} \textit{Jean-Pierre Serre},
	Topics in Galois theory. Second edition. With notes by Henri Darmon. Research Notes in Mathematics, 1. A K Peters, Ltd., Wellesley, MA 2008.
	
	\bibitem{Silverman} \textit{Joseph H. Silverman},
	The arithmetic of elliptic curves, 2nd ed.,
	Springer, Dordrecht 2009.





\end{thebibliography}
\end{document}